\documentclass[10pt]{amsart}
\usepackage{amsmath}
\usepackage{amsthm}
\usepackage[initials, nobysame]{amsrefs}
\usepackage{geometry}                
\usepackage{graphicx}
\usepackage{xcolor}
\usepackage{hyperref}
\usepackage{amssymb}
\usepackage{epstopdf}
\usepackage{tikz,pgf}

\theoremstyle{plain}
\newtheorem{theorem}{Theorem}[section]
\newtheorem{lemma}[theorem]{Lemma}

\newtheorem{proposition}[theorem]{Proposition}

\newtheorem{corollary}[theorem]{Corollary}
\theoremstyle{definition}

\newtheorem{definition}[theorem]{Definition}

\newtheorem{question}[theorem]{Question}
\newtheorem{construction}[theorem]{Construction}

\theoremstyle{remark}

\newtheorem{remark}[theorem]{Remark}

\renewcommand{\epsilon}{\varepsilon}
\newcommand{\NN}{\ensuremath{\mathbb{N}}}

\newcommand{\s}[1]{\ensuremath{\mathcal{#1}}}
\newcommand{\sZ}{\s{Z}}
\newcommand{\sX}{\s{X}}
\newcommand{\sE}{\s{E}}
\newcommand{\sF}{\s{F}}
\newcommand{\sU}{\mathcal{U}}
\newcommand{\sV}{\mathcal{V}}
\newcommand{\sW}{\mathcal{W}}
\newcommand{\sY}{\mathcal{Y}}

\DeclareMathOperator{\as}{asdim}

\DeclareMathOperator{\diam}{diam}
\DeclareMathOperator{\supp}{supp}

\begin{document}
\title{Coarse property {C} and decomposition complexity}
\author{G.~ Bell}
\address{Department of Mathematics and Statistics, University of North Carolina at Greensboro, Greensboro, NC 27412, USA}
\email{gcbell@uncg.edu}
\author{D.~Moran}
\address{Mathematics Department, Guilford College, 5800 West Friendly Avenue, Greensboro, NC 27410, USA}
\email{morands@guilford.edu}

\author{A.~Nag\'orko}
\address{Faculty of Mathematics, Informatics, and Mechanics, University of Warsaw, Banacha 2, 02-097 Warszawa, Poland}
\email{amn@mimuw.edu.pl}
\date{\today}                                           

\begin{abstract}
The coarse category was established by Roe \cite{Ro03} to distill the salient features of the large-scale approach to metric spaces and groups that was started by Gromov \cite{Gr93}. In this paper, we use the language of coarse spaces to define coarse versions of asymptotic property C \cite{Dr00} and decomposition complexity \cite{GTY1}. We prove that coarse property C implies coarse property A; we also show that these coarse versions share many of the features of their metric analogs such as preservation by products or unions. 
\end{abstract}



\maketitle

\dedicatory{This paper is dedicated to the memory of Alex Chigogidze, who introduced Bell and Nag\'orko to each other and who led the development of the Ph.D. program at the University of North Carolina at Greensboro.}

\section{Introduction}

The large-scale approach to metric spaces began with Gromov's seminal paper \cite{Gr93} concerning metric properties of finitely generated groups. Because finitely generated groups have a natural metric structure that is unique up to quasi-isometry, the class of finitely generated groups is a natural setting for such large-scale invariants. One of the properties he defines is asymptotic dimension. Asymptotic dimension and the large-scale approach to groups and metric spaces rose to prominence following two papers of Guoliang Yu, in which it was shown that the famous Novikov higher signature conjecture holds for finitely generated groups with finite asymptotic dimension \cite{Yu1} and the coarse Baum-Connes conjecture holds for spaces that admit a uniform embedding into Hilbert space \cite{Yu2}.

Asymptotic dimension has the advantage of being very easy to define and, since it is analogous to the classical covering dimension, its properties are readily studied. However, Yu's second result concerned spaces that admit a coarsely uniform embedding into Hilbert space (see \cite{BNSWW} for a nice introduction to many such properties). Although spaces with finite asymptotic dimension admit such an embedding, there are spaces with infinite asymptotic dimension that also admit such embeddings. Thus began an industry devoted to investigating large-scale invariants that are weaker than finite asymptotic dimension, but are strong enough to guarantee a uniform embedding into Hilbert space. Among these invariants are Dranishnikov's asymptotic property C \cite{Dr00}, Yu's property A \cite{Yu2}, finite decomposition complexity \cite{GTY1}, and straight finite decomposition complexity \cite{DZ}. The primary goal of the current work is to formulate versions of these invariants in the coarse category.

The coarse category was introduced by Roe \cite{Ro03} to abstractly quantify important features of metric spaces viewed on the large scale. The coarse category is an abstraction of the large-scale geometry of metric spaces in the same way that topology is an abstraction of our understanding of continuity and open balls in metric spaces. Thus, it is hoped that by defining the notions of property C and decomposition complexity in terms of the coarse structure only, we can better understand these properties.

The paper is organized as follows. In the next section we give precise definitions of the metric versions of asymptotic dimension, asymptotic property C, property A, and the various notions of decomposition complexity. In the third section we define our coarse property C based on the definition of coarse asymptotic dimension given in \cites{Ro03,Gra05}. The main result of this section is Theorem \ref{cpCcpA}, which says that coarse property C implies Roe's coarse property A \cite{Ro03}. This result is a natural extension of Dranishnikov's result that asymptotic property C implies property A for metric spaces \cite{Dr00}. Moreover we show that coarse property C is preserved by certain types of infinite unions and all finite unions (Theorem \ref{cPCUnion}).

In the fourth section we define various notions of coarse decomposition complexity. We relate these notions to each other and discuss how they are related to coarse asymptotic dimension and coarse property C. Much of this section follows the general outline from \cite{Guentner-Survey} regarding permanence properties in coarse geometry. The results of this section are permanence results of the following type: suppose $X$ is a coarse space built in some way from a coarse space $Y$ with a coarse property $\mathcal{P}$. Then, $X$ inherits that coarse property $\mathcal{P}$. For example, if $X=Y\times Y'$ where $Y$ and $Y'$ are coarse spaces with finite coarse decomposition complexity, then $X$ will have finite coarse decomposition complexity. The main results here are Theorem \ref{fibering}, and that these notions of decomposition complexity are preserved by finite unions, some infinite unions, and cartesian products.

\section{Preliminaries}

In this section we begin by describing various properties of metric spaces that are linked to the question of whether the space admits a coarsely uniform embedding into Hilbert space. 

In the second part of this section we describe Roe's coarse category, which is a natural generalization of the large-scale approach to metric spaces.

\subsection{Coarse Invariants of Metric Spaces}
Let $(X,d_X)$ and $(Y,d_Y)$ be metric spaces. Recall that a map $f:X\to Y$ is called \textit{proper} if the preimage of every compact set is compact. A metric is called \textit{proper} if the distance function is a proper map, i.e., if closed balls are compact.
A function $f:X\to Y$ is called {\em uniformly expansive} if there is a non-decreasing $\rho_2:[0,\infty)\to[0,\infty)$ such that 
\[d_Y(f(x),f(x'))\le\rho_2(d_X(x,x')).\]

The function $f:X\to Y$ is called {\em effectively proper} if there is some proper, non-decreasing $\rho_1:[0,\infty)\to[0,\infty)$ such that
\[\rho_1(d_X(x,x'))\le d_Y(f(x),f(x')).\]
The function $f:X\to Y$ is called a {\em coarsely uniform embedding} if there exist functions $\rho_1,\rho_2:[0,\infty)\to[0,\infty)$ such that, \[\rho_1(d_X(x,x'))\le d_Y(f(x),f(x'))\le \rho_2(d_X(x,x'))\] and $\rho_1\to\infty$. The spaces $X$ and $Y$ are said to be {\em coarsely equivalent} if there is a coarsely uniform embedding of $X$ to $Y$ and there is some $R>0$ so that for each $y\in Y$ there is some $x\in X$ so that $d_Y(f(x),y)\le R$. When the $\rho_i$ can be taken to be linear, $f$ is called a {\em quasi-isometric embedding} and the corresponding equivalence is {\em quasi-isometry}. 

Gromov defines the asymptotic dimension as a large-scale analog of the usual Lebesgue covering dimension \cite{Gr93}. 
Let $R>0$ be a (large) real number. A collection $\sU$ of subsets of the metric space $X$ is said to be \textit{uniformly bounded} if there is a uniform bound on the diameter of the sets in $\sU$; a collection $\sU$ is said to be \textit{$R$-disjoint} if, whenever $U\neq U'$ are sets in $\sU$, then $d(U,U')>R$, where $d(U,U')=\inf\{d(x,x')\mid x\in U, x'\in U'\}$. A family that is uniformly bounded and $R$-disjoint is said to be $0$-{\em dimensional on} $R$-{\em scale}. 
 We say that the {\em asymptotic dimension} of the metric space $X$ does not exceed $n$ and write $\as X\le n$ if for each (large) $R>0$, $X$ can be written as a union of $n+1$ sets with dimension $0$ at scale $R$. 
There are several other useful formulations of the definition (see \cite{BD-AD}) but we shall content ourselves with this one.

As mentioned in the introduction, this property rose to prominence following the work of G.~Yu \cite{Yu1} showing that the famous Novikov higher signature conjecture holds for finitely generated groups with finite asymptotic dimension. Following that, Yu generalized this result to bounded geometry metric spaces that admit a uniform embedding into Hilbert space \cite{Yu2}. In that paper, Yu defined property A for discrete metric spaces as a generalization of amenability of groups. A discrete metric space $X$ has {\em property A} if for any $r>0$ and any $\epsilon>0$, there is a collection of finite subsets $\{A_x\}_{x\in X}$, where $A_x\subset X\times\NN$, so that 
\begin{enumerate}
	\item $(x,1)\in A_x$ for each $x\in X$;
	\item for every pair $x$ and $y$ in $X$ with $d(x,y)<r$, $\frac{|A_x\Delta A_y|}{|A_x\cap A_y|}<\epsilon$; and 
	\item there is some $R$ so that for every $n$ such that $(y,n)\in A_x$, $d(x,y)\le R$.
\end{enumerate}
A space that has property A admits a uniform embedding into Hilbert space. Higson and Roe \cite{HR} showed that discrete metric spaces with bounded geometry and finite asymptotic dimension have property A.

Dranishnikov defined the notion of {\em asymptotic property C} for metric spaces in his work on asymptotic topology as an asymptotic analog of Haver's property C \cite{Dr00}.
A metric space $X$ has {\em asymptotic property C} if for any number sequence $R_1\le R_2\le R_3\le\cdots$ there is a finite sequence of uniformly bounded families of open sets $\{\mathcal{U}_i\}^k_{i=1}$ such that the union $\bigcup_{i=1}^k\mathcal{U}_i$ is a covering of $X$ and every family $\mathcal{U}_i$ is $R_i$-disjoint.

Asymptotic property C is easily seen to be a coarse invariant and any metric space with finite asymptotic dimension will have asymptotic property C. Dranishnikov showed that a discrete metric space with bounded geometry and asymptotic property C also has property A \cite[Theorem 7.11]{Dr00}. 

Guentner, Tessera, and Yu \cites{GTY1,GTY2} defined another coarse invariant of groups that is applicable when the asymptotic dimension is infinite: finite decomposition complexity. Following this, Dransihnikov and Zarichnyi defined a related notion: straight finite decomposition complexity. By way of notation, we write $A=\displaystyle\sqcup_{\hbox{\tiny$R$-disjoint}}A_i$ to mean that the subset $A$ can be decomposed as a union of sets $A_i$ in such a way that $d(A_i,A_j)>R$ whenever $i\neq j$.
Let $\sU$ and $\sV$ be families of subsets of a fixed metric space $X$. For a positive $R$, we say that $\sU$ is $R$-decomposable over $\sV$ and write $\sU\xrightarrow{R}\sV$ if for any $U\in\sU$ one can write 
\[U=U^0\cup U^1\, \hbox{ where } U^i=\bigsqcup_{\hbox{\footnotesize$R$-disjoint}} V^{i}_j,\hbox{ for i=0,1,}\]
where the sets $V^{i}_j\in\sV$.

We begin by describing the metric decomposition game for a metric space $X$. In this game two players take turns. First, Player 1 asserts a number $R_1$. Player 2 responds by finding a metric family $\sY^1$ and a $R_1$-decomposition of $\{X\}$ over $\sY^1$. Then, Player 1 selects a number $R_2$ and Player 2 again finds a family $\sY^2$ and an $R_2$-decomposition of $\sY^1$ over $\sY^2$. The game ends if Player 2 can respond to an assertion of Player 1 with a decomposition over a family that consists of uniformly bounded subsets at some finite step. The metric space $X$ is said to have {\em finite composition complexity} or FDC, if there is a winning strategy for Player 2 in the metric decomposition game for $X$, see \cite{GTY1}.

Let $\sX$ be a family of metric spaces. We say that the metric family $\sX$ has {\em straight finite decomposition complexity} sFDC if for every sequence $R_1\le R_2\le \cdots$ there exists an $n$ and metric families $\sY^i$ ($i=1,2,3,\ldots, n)$ so that with $\sX=\sY^0$, $\sY^{i-1}\xrightarrow{R_i}\sY^i$ for $i=1,2,3,\ldots, n$, and such that $\sY^n$ is uniformly bounded, see \cite{DZ}. The metric space $X$ will be said to have sFDC if the family $\{X\}$ does. It is clear that by restricting the families, this property can be seen to pass to subsets. Furthermore, metric spaces with finite asymptotic dimension have finite decomposition complexity \cite[Theorem 4.1]{GTY1} and spaces with asymptotic property C have straight finite decomposition complexity \cite{DZ}. Since finite asymptotic dimension implies asymptotic property C, it is also true that metric spaces with finite asymptotic dimension have straight finite decomposition complexity. Finally, since in the challenges for disjointness in the decomposition game for sFDC are revealed at the outset, any space that has FDC will have sFDC.

In \cite{Dydak15}, Dydak introduced the notion of countable asymptotic dimension as follows. A metric space $X$ is said to have {\em countable asymptotic dimension} if there is a sequence of positive integers $n_i$ ($i\ge 1$) such that for any sequence of positive real numbers $R_i$, there is a sequences $\sV_i$ of families of subsets of $X$ such that the following conditions hold:
\begin{enumerate}	
	\item $\sV_1=\{X\}$;
	\item each element $U\in \sV_i$ can be expressed as a union of at most $n_i$ families from $\sV_{i+1}$ that are $R_i$ disjoint for each $i$; and 
	\item at least one of the families $\sV_i$ is uniformly bounded. 
\end{enumerate}

It is shown by Dydak and Virk in \cite[Theorem 8.5]{DV} that a metric space $X$ has sFDC if and only if it has countable asymptotic dimension.

\subsection{The Coarse Category}

We describe the coarse category following Grave's dissertation \cite{Gra05}. Let $X$ be a set. A coarse structure on $X$ is a collection of subsets $\sE$ of $X\times X$ called {\it entourages} that satisfies the following axioms:
\begin{enumerate}
\item A subset of an entourage is an entourage;
\item a finite union of entourages is an entourage;
\item the diagonal $\Delta_X:=\{(x,x)\mid x\in X\}$ is an entourage;
\item the inverse $E^{-1}$ of an entourage is an entourage, where $E^{-1}:=\{(y,x)\mid (x,y)\in E\}$; and
\item the composition of two entourages is an entourage, where $E_1\circ E_2:=\{(x,z)\mid \exists y\in X, (x,y)\in E_1, (y,z)\in E_2\}$.
\end{enumerate}

The collection of elements of an entourage will sometimes be called the {\it controlled sets}. An entourage will be called {\it symmetric} if $E=E^{-1}$.

We call a subset $B\subset X$ {\em bounded} if $B\times B\in\sE$. For an entourage $E$, we say that the family $\sU$ of subsets of $X$ is $E$-{\it disjoint} if whenever $A, B\in \sU$ with $A\neq B$, then $A\times B\cap E=\emptyset$. By way of notation, denote by $\Delta_\sU$ the union $\Delta_\sU=\bigcup_{U\in\sU}\left(U\times U\right)$. We say that a family $\sU$ of subsets of a coarse space $(X,\sE)$ is {\it uniformly bounded} if $\Delta_\sU\in\sE$. By way of notation, let $E[A]:=\{x\in X \mid (x,a)\in E\hbox{ for some }a\in A\}$.  Grave shows \cite[Proposition 3.10]{Gra05} that a cover $\sU$ is uniformly bounded if and only if there is an entourage $D$ such that for every $U\in\sU$ there is an $x\in X$ such that $U\subset D[\{x\}]$.

Let $(X,\sE)$ and $(Y,\sF)$ be coarse spaces. The following definitions come from \cite{Ro03}. A map $f:X\to Y$ is said to be {\em proper} if the inverse image of every bounded subset of $Y$ is bounded in $X$. The map $f:X\to Y$ is said to be {\em bornologous} if for each entourage $E\in\sE$, we have $(f\times f)(E)\in\sF$. The map is said to be {\em coarse} if is is both proper and bornologous. A map $f:X\to Y$ is said to be a {\em coarsely uniform embedding} if it is coarse (i.e. proper and bornologous) and $(f^{-1}\times f^{-1})(F)\in\sE$ for every $F\in\sF$. Two maps $f,f':X\to Y$ are said to be close if $\{(f(x),f'(x))\mid x\in X\}$ is an entourage in $Y$. The spaces $X$ and $Y$ are said to be {\it coarsely equivalent} if there are coarse maps $f:X\to Y$ and $g:Y\to X$ such that the compositions (in both directions) are close to the identity.

\section{Coarse Property {C}}
Roe (and later Grave) defined a coarse analog of asymptotic dimension. In order to be clear, we will write $\as_C$ for the coarse asymptotic dimension (although it should always be clear from context).

\begin{definition}[\cite{Ro03, Gra06}]
A coarse space $(X,\sE)$ satisfies the inequality $\as_C X \leq n$ if for any entourage $E\in\sE$ there exists a finite sequence $\sU^0, \sU^1, \ldots, \sU^n$ of families of subsets of $X$ so that \begin{enumerate}
	\item $\sU=\bigcup_{i=0}^n \sU^i$ covers $X$;
	\item each $\sU^i$ is uniformly bounded; and
	\item each $\sU^i$ is $E$-disjoint.
\end{enumerate}
\end{definition}

If we translate the notions from the category of metric spaces to coarse spaces in the sense of Roe \cite{Ro03}, we obtain the following definition, which we call {\em coarse property C}. See also \cite{Gra06}.

\begin{definition}\label{cpC} A coarse space $(X,\sE)$ has \emph{coarse property C} if for any sequence $L_1\subset L_2\subset L_3\subset \cdots$ of entourages there is an $n$ and a finite sequence $\sU^1, \sU^2, \ldots, \sU^n$ of families of subsets of $X$ so that \begin{enumerate}
	\item $\sU=\bigcup_{i=1}^n \sU^i$ covers $X$;
	\item each $\sU^i$ is uniformly bounded; and
	\item each $\sU^i$ is $L_i$-disjoint.
\end{enumerate}
\end{definition}

It follows from the definitions that for a coarse space $(X,\sE)$, $\as_C X<\infty$ implies $X$ has coarse property C. It is also easy to see that coarse property C is inherited by subsets by restriction.

\subsection{Basic Properties}


%

\begin{proposition} Coarse property C is a coarse invariant.
\end{proposition}

\begin{proof}
Let $(X,\sE)$ and $(Y,\sF)$ be coarsely equivalent coarse spaces. Then, there are coarsely uniform embeddings $f:X\to Y$ and $g:Y\to X$ whose compositions are close to the identities. We show that if $f:X\to Y$ is a coarsely uniform embedding and $(Y,\sF)$ has coarse property C, then $(X,\sE)$ will have coarse property C. The subset $f(X)\subset Y$ inherits coarse property C, so to ease notation, we will replace $Y$ with the image $f(X)$ and let $\sF$ denote the restriction to this image.  Let $L_1\subset L_2\subset L_3\subset \cdots$ be a sequence of entourages in $\sE$.  Then with $K_i=(f\times f)(L_i)$, we find $K_1\subset K_2\subset K_3\subset \cdots$ is a sequence of entourages in $\sF$.

Therefore, since $(Y,\sF)$ has coarse property C, there is a finite sequence $\sU^1, \sU^2, \ldots, \sU^n$ of families of subsets of $Y$ satisfying the conditions of coarse property C.  Let $\sV^i = \{f^{-1}(U) \mid U\in\sU^i\}$.  Since $\sU=\bigcup_{i=1}^n \sU^i$ covers $Y$, we have that $\sV=\bigcup_{i=1}^n \sV^i$ covers $X$.  

Now, denoting $\bigcup_{V\in\sV} V\times V$ by $\Delta_\sV$, we have that $\Delta_\sV = \bigcup_{U\in\sU} f^{-1}(U)\times f^{-1}(U) = (f\times f)^{-1}(\bigcup_{U\in\sU} U\times U) = (f\times f)^{-1}(\Delta_\sU)$.  By assumption $\sU$ is uniformly bounded and so $\Delta_\sU$ is an entourage.  Since $f$ is a coarsely uniform embedding, we have that  $(f\times f)^{-1}(\Delta_\sU) = \Delta_\sV$ is an entourage and therefore $\sV$ is uniformly bounded.

It remains to show that $\sV^i$ is $L_i$-disjoint.  Let $A,B \in \sV^i$, with $A\neq B$.  Then $A = f^{-1}(A')$ for some $A'\in \sU^i$ and $B = f^{-1}(B')$ for some $B'\in \sU^i$, with $A' \neq B'$.  Then, $\left(A\times B\right) \cap L_i\subset (f\times f)^{-1}\left(\left(A'\times B'\right) \cap K_i\right) = (f\times f)^{-1}(\emptyset) = \emptyset$ since $\sU^i$ is $K_i$-disjoint. We conclude that $\sV^i$ is $L_i$-disjoint and therefore $X$ has coarse property C.
\end{proof}

Instead of assuming a coarsely uniform embedding between $X$ and $Y$, one may consider the extent to which property C is preserved by other types of coarse maps. In the metric case this is thoroughly investigated in \cite[Theorem 6.2 and Theorem 6.3]{DV}. In the case of coarse asymptotic dimension, such investigation was taken up by Austin in \cite{Austin}.

A metric space $X$ can be endowed with a natural coarse structure: the so-called bounded coarse structure that is derived from the metric on $X$. This structure is the one for which the entourages $\sE$ consist of those $E\subset X$ for which $\sup\{d(x,x')\mid (x, x')\in E\}$ is finite. We show that a metric space has asymptotic property C if and only if the corresponding coarse space has coarse property C. Thus, ours is in some sense the ``correct'' way to define coarse property C.

\begin{proposition} 
Let $(X,d)$ be a metric space. Let $\mathcal{E}$ denote the bounded coarse structure on $X$. Then $(X, d)$ has asymptotic property C if and only if $(X, \mathcal{E})$ has coarse property C.
\end{proposition}

\begin{proof} 
Suppose first that $(X,d)$ has asymptotic property C. Let $L_1\subset L_2\subset\cdots$ be a sequence of controlled sets. For each $i$, put $R_i=\sup\{d(x,x')\mid (x,x')\in L_i\}$. Then each $R_i$ is finite, by the definition of the bounded coarse structure and moreover $R_1\le R_2\le \cdots$.

Since $(X,d)$ has asymptotic property C, there are families $\mathcal{U}_1, \mathcal{U}_2, \cdots, \mathcal{U}_k$ that cover $X$, that consist of uniformly bounded sets, and that are $R_i$-disjoint ($i=1,2,\dots,k$). 

We need to show that the $\mathcal{U}_i$ are coarsely uniformly bounded and $L_i$-disjoint.

The collection $\mathcal{U}_i$ is coarsely uniformly bounded if and only if $\Delta_{\mathcal{U}_i}=\bigcup_{\alpha}U^i_\alpha\times U^i_\alpha$ is in $\mathcal{E}$. But, $\Delta_{\mathcal{U}_i}\in\mathcal{E}$ if and only if 
\[\sup\{d(x,y)\mid (x,y)\in \Delta_{\mathcal{U}_i}\}<\infty,\] which is implied by our assumption that the family has uniformly bounded diameter,  \[\sup_{\alpha}\{\operatorname{diam}(U^i_\alpha)\}<\infty.\] 

Next, to show that the $\mathcal{U}_i$ are $L_i$-disjoint, we must show that $\left(U^i_\alpha\times U^i_\beta\right)\cap L_i=\emptyset$ whenever $U^i_\alpha\neq U^i_\beta$. Suppose that $a\in U^i_\alpha$ and $b\in U^i_\beta$ and $(a,b)\in L_i$. Then, we have $d(a,b)\le R_i$, which contradicts the fact that the family $\mathcal{U}_i$ is $R_i$-disjoint. Thus, asymptotic property C implies coarse property C. 

On the other hand, suppose now that $(X,\mathcal{E})$ has coarse property C and let $R_1\le R_2\le\cdots$ be given. Define a sequence of controlled sets $L_i$ as follows:
\[L_i=\{(x,y)\in X\times X\mid d(x,y)\le R_i\}.\] Using this sequence, we find a cover of $X$ by uniformly bounded $\mathcal{U}_1,\mathcal{U}_2,\ldots,\mathcal{U}_k$, where each $\mathcal{U}_i$ is $L_i$-disjoint. 

If $\sU_i$ is $L_i$-disjoint and if $U^i_\alpha\neq U^i_\beta$, then $(U^i_\alpha\times U^i_\beta)\cap L_i=\emptyset.$ Thus, if $(a,b)\in U^i_\alpha\times U^i_\beta$, then $(a,b)\notin L_i,$ i.e., $d(a,b)>R_i$, which is to say the family $\sU_i$ is $R_i$-disjoint.

On the other hand, if each family $\sU_i$ is coarsely uniformly bounded, then $\Delta_{\sU_i}=\bigcup\left(U^i_\alpha\times U^i_\alpha \right)\in\sE$. Thus, $\sup_{\alpha}\{\diam(U^i_\alpha)\}<\infty$. Thus, $\sU_i$ is metrically uniformly bounded. Thus, coarse property C implies asymptotic property C.
\end{proof}


\subsection{Coarse property C implies coarse property A}

In the metric setting, property A can be formulated in several different ways \cite{Willett}. Roe defined property A in terms of coarse structure by translating a notion of property A that is equivalent to Yu's original definition for discrete metric spaces with bounded geometry \cite{Ro03}*{Definition 11.35} (see also \cite{Sato}). His definition involves the existence of maps from the space $X$ to $\ell^2(X)$, whereas we use $\ell^1$ for simplicity (compare \cite{Dr06}*{Proposition 3.2}). We prove that coarse property C implies coarse property A.
Indeed, in the special case of metric spaces, we recover the result of Dranishnikov \cite{Dr00}: that discrete metric spaces with bounded geometry and asymptotic property C have asymptotic property A. Indeed one point of this line of reasoning was to determine to what extent Dranishnikov's arguments could be applied in the coarse setting.

\begin{definition} A coarse space $(X,\sE)$ will be said to have {\em coarse property A} if for each $\varepsilon > 0$
and for each $E \in \sE$ there exists a map $a:X\to \ell^1(X)$, expressed as $a:x\mapsto a_x$ such that:
\begin{enumerate}
	\item $\|a_x\|_1=1$ for all $x\in X$;
	\item if $(x, y) \in E$, then $|| a_x - a_y ||_1 < \varepsilon$; 
	\item there exists $S\in\sE$ such that for each $x \in X$, $\supp a_x \subset S[x]$.
\end{enumerate}
\end{definition}

To begin, we have a sequence of propositions that enable us to use entourages to give some measure of distance in a coarse space. By way of notation, for an entourage $E$ in a coarse space $(X,\sE)$, write $E^0$ to mean the entourage $\Delta$. For a positive integer $n$ and an entourage $E$, we write $E^n$ to mean the $n$-fold composition $E\circ E\circ \cdots \circ E$. In particular, superscripts on entourages will not denote indices. 

\begin{proposition}\label{Distance}
Let $(X,\mathcal{E})$ be a coarse space and fix a symmetric entourage $E$. Define $D: X\times X \to \mathbb{Z} \cup \{+\infty\}$ by $D(x,y) = \min\{k\geq 0 \mid(x,y)\in  E^k\}$. Then we have that 
\begin{enumerate}
	\item $D$ is symmetric, 
	\item $D(x,y) = 0$ iff $x = y$,
	\item $D(x,y) \leq D(x,z) + D(z,y)$,
    \item for every $w, z \in X$ and each $A \subset X$ we have      \[ |D(w,A) - D(z, A)| \leq D(w,z),\]
    where $D(x, A) = \inf \{ d(x,a) \colon a \in A \}$.
\end{enumerate}
\end{proposition}
\begin{proof}We note that from the definition, if $D(x,y) = k$, then $(x,y)\in E^{k}$.
Since $E$ is symmetric, it follows that $E^{k}$ is symmetric. So if $(x,y) \in E^{k}$ then $(y,x)\in E^{k}$ and thus $D(x,y) = D(y,x)$. 

The second statement is obvious. 

Suppose now that $D(x,z)$ and $D(z,y)$ are finite, say $D(x,z) = k$ and $D(z,y)=l$. Thus, $(x,z)\in E^{k}$ and $(z,y)\in E^{l}$. We see then, that $(x,y) = (x,z)\circ (z,y) \in E^{k}\circ E^{l} = E^{k+l}$ (we use associativity of $\circ$ here) and thus $D(x,y) \leq k+l=D(x,z)+D(z,y)$.

If $D(x,y)=+\infty$, then there is no $k$ so that $(x,y)\in E^k$, so at least one of $D(x,z)$ and $D(z,y)$ must also be infinite. Thus, the inequality holds in either case. 

In 1.--3. we proved that $D$ is a metric (which can take infinite values) and 4. is well known to hold for metric functions.
\end{proof}

\begin{proposition}\label{phi}
  Let $(X, \mathcal{E})$ be a coarse space and let $E \in \mathcal{E}$ be a symmetric entourage with $\Delta \subset E$.
  Let $n \in \mathbb{N}, n > 0$.
  Let $\sU = \{ U_j \}$ be an $E^{n}$-disjoint uniformly bounded family of subsets of $X$.
  Define $\varphi_j = \varphi_{j,n}^\sU \colon X \to \mathbb{R}$ by the formula
  \[
    \varphi_j(x) = \max \{ 0, \frac n4 - D(x, U_j) \}
  \]
  and $\varPhi = \varPhi_n^\sU \colon X \to \mathbb{R}$ by the formula
  \[
    \varPhi(x) = \sum_j \varphi_j(x).
  \]
  Then
  \begin{enumerate}
  \item for each $x \in X$ there exists at most one $j$ such that $\varphi_j(x) \neq 0$,
  \item The map $\varPhi$ is $1$-Lipschitz with respect to $D$, i.e. for every $w, z \in X$ we have
  \[
    |\varPhi(w) - \varPhi(z) | \leq D(w,z).
  \]
  \end{enumerate}
\end{proposition}
\begin{proof}
1. Suppose to the contrary that there there are $j\neq j'$ with both $\phi_j(x)\neq 0$ and $\phi_{j'}(x)\neq 0$. Put $d=D(x,U_j)$ and $d'=D(x,U_{j'}).$ Then $d < \frac{n}{4}$ and $d' < \frac{n}{4}$. Therefore, there exist $z_j\in U_j$ and $z_{j'}\in U_{j'}$ so that $(x,z_j)\in E^{d}$ and $(x,z_{j'})\in E^{d'}$. Together these imply that $(z_j,z_{j'})\in E^{d+d'}$ and $d+d' < \frac n2 \leq n$. Thus, $(z_j,z_{j'})\in E^{n}$, which contradicts the fact that the family $\sU$ is $E^{n}$-disjoint.

2. First consider the case $\varPhi(w) \neq 0$ and $\varPhi(z) \neq 0$. By 1., there exists a pair $j$ and $j'$ such that $\varPhi(w) = \varphi_j(w)$ and $\varPhi(z) = \varphi_{j'}(z)$.
Then, by the definition, we have
\[
  |\varphi_j(w) - \varphi_{j'}(z)| = |D(z, U_{j'}) - D(w, U_j)|. \]
If $j = j'$, then $|D(z, U_{j'}) - D(w, U_j)| \leq D(w,z)$ by Lemma~\ref{Distance}.
If $j \neq j'$, then $|D(z, U_{j'}) - D(w, U_j)| \leq \frac n4 + \frac n4 = \frac n2 \leq D(w,z)$.

Now consider the case $\varPhi(w) \neq 0$, $\varPhi(z) = 0$.
By 1., there exists $j$ such that $\varPhi(w) = \varphi_j(w)$.
Then, by the definition, we have
\[
  |\varPhi(w) - \varPhi(z)| = | \frac n4 - D(w, U_j) - 0 |
  \leq |D(z,U_j) - D(w, U_j)| \leq D(w,z).
\]
\end{proof}

Next, we will construct a function into $\ell^1(X)$ related to coarse property A.

\begin{construction}\label{constr}
  We construct a map $b^n_x \in \ell^1(X)$.
  Let $(X, \mathcal{E})$ be a coarse space, let $E \in \mathcal{E}$, $\Delta \subset E$, $E = E^{-1}$. 
  Let $n \in \mathbb{N}$, $n > 1$.
  Let $\sU^1, \sU^2, \ldots, \sU^k$ be a sequence families of uniformly bounded subsets of $X$.
  Denote $\sU^i = \{ U^i_j \}_j$ and assume that $\sU^i$ is $E^{n^i}$-disjoint.
  Assume that $\sU = \bigcup_i \sU^i$ is a cover of $X$.
  For each $i,j$ pick $x^i_j \in U^i_j$.
  Let $D^i_x = \{ x^i_j \colon D(x, U^i_j) < \frac{n^i}4 \}$.
  We define $b^n_x \in \ell^1(X)$ by the formula
  \[
    b^n_x(y) = \sum_{i=1}^k n^{k-i+1} \varPhi^{\sU^i}_{n^i}(x) \cdot \chi_{D^i_x}(y),
  \]
  where $\chi_{D^i_x}$ denotes the characteristic function of the set $D^i_x$.
\end{construction}

\begin{lemma}\label{bnx}
  If $b^n_x$ is defined as in Construction~\ref{constr}, then
  \begin{enumerate}
    \item $\frac {n^{k+1}}4 \leq || b^n_x ||_1 < \infty$,
    \item there exists $S \in \mathcal{E}$ such that for each $x \in X$ we have $\supp b^n_x \subset S[x]$,
    \item $||b^n_z - b^n_w ||_1 \leq \frac{n(n^k-1)}{n-1} D(z,w).$
  \end{enumerate}
\end{lemma}
\begin{proof}
1. Fix $x \in X$. Since $\sU$ covers $X$, there exists $U^i_j$ such that $x \in U^i_j$. We have
  \[
    ||b^n_x||_1 \geq n^{k-i+1}\cdot \varPhi^{\sU^i}_{n^i}(x) \cdot \chi_{D_x^i}(x^i_j) = n^{k-i+1} \cdot \left(\frac {n^i}4 - D(x, U^i_j) \right)\cdot 1 = \frac {n^{k+1}} 4.
  \]
  
2. We have $\supp b^n_x = \bigcup_{i=1}^k D^i_x$.
If $x^i_j \in D^i_x$, then $D(x, U^i_j) < \frac{n^i}4$.
By the assumptions, for each $i$ there exists $S_i \in \mathcal{E}$ such that $\sU^i$ is $S_i$-bounded.
Then $(x, x^i_j) \in E^{\lceil \frac n4 \rceil} \circ S_i[x]$.
Let $S = E^{n^k} \circ \bigcup_{i=1}^k S_i$.
Clearly, $S$ is a controlled set and $\sU$ is bounded by $S$.

3. We have
\[
  ||b^n_w - b^n_z||_1 = \left\| \sum_{i=1}^k n^{k-i+1} \left( \varPhi^{\sU^i}_{n^i}(w)\chi_{D^i_w} - \varPhi^{sU^i}_{n^i}(z)\chi_{D^i_z}
  \right) \right\|_1 \leq
\]
\[
  \leq \sum_{i=1}^k n^{k-i+1} ||\varPhi(w)\chi_{D^i_w} - \varPhi(z)\chi_{D^i_z} ||_1
\]

We will prove that the norm in the last expression is bounded by $D(w,z)$. We consider three cases.
\begin{enumerate}
\item[a)] If $D^i_w = D^i_z \neq \emptyset$, then $||\varPhi(w)\chi_{D^i_w} - \varPhi(z)\chi_{D^i_z} ||_1 = | \varPhi(w) - \varPhi(z) | \leq D(w,z)$ by Proposition~\ref{phi} and the fact that sets $D^i_w$ and $D^i_z$ contain a single point.
\item[b)] If $D^i_w \neq D^i_z$ and both sets are non-empty, then $D(w, z) \geq \frac{n^i}2$. Since $\varPhi(w), \varPhi(z) \leq \frac{n^i}4$, we have $||\varPhi(w)\chi_{D^i_w} - \varPhi(z)\chi_{D^i_z} ||_1 \leq \frac{n^i}2 \leq D(w,z)$.
\item[c)] If $D^i_w = \emptyset$, then $\varPhi(w) = 0$. Then $||\varPhi(w)\chi_{D^i_w} - \varPhi(z)\chi_{D^i_z} ||_1 = 
 ||\varPhi(w)\chi_{D^i_z} - \varPhi(z)\chi_{D^i_z} ||_1$
 and we argue as in case a).
\end{enumerate}

Therefore 
\[
\sum_{i=1}^k n^{k-i+1} ||\varPhi(w)\chi_{D^i_w} - \varPhi(z)\chi_{D^i_z} ||_1 \leq \sum_{i=1}^k n^{k-i+1} \cdot D(w,z) = \frac{n(n^k-1)}{n-1}D(w,z).
\]
\end{proof}

Finally, we conclude with the main result of this section.

\begin{theorem}\label{cpCcpA}
  If a coarse space $(X, \mathcal{E})$ has coarse property C, then it has coarse property A.
\end{theorem}
\begin{proof}
  Fix $\varepsilon > 0$ and $E \in \mathcal{E}$ such that $\Delta \subset E, E = E^{-1}$.
  Take $n \in \mathbb{N}$ such that $0 < \frac 8{n-1} < \varepsilon$.
  By property C of $X$, there exist a finite sequence $\sU^1, \sU^2, \ldots, \sU^k$ of families of subsets of $X$ such that each $\sU^i$ is $E^{n^i}$-disjoint, uniformly bounded, and such that $\sU = \bigcup_{i=1}^{k} \sU^i$ is a cover of $X$.
  Let
  \[
    a_x = \frac{b^n_x}{||b^n_x||_1}.
  \]
  
  Clearly from the definition, $||a_x||_1 = 1$. This is the first condition of coarse property A.

  We will show that if $(x,y) \in E$, then $||a_x - a_y|| < \varepsilon$. This is the second condition of coarse property A. Let $(x,y) \in E$. Then $D(x,y) \leq 1$.
From the inequality
\[
  || \frac u {|| u ||} - \frac v {|| v||}|| \leq \frac 1 {||u||} \cdot 2 \cdot || u - v ||
\]
we have
\[
  || \frac{b^n_x}{||b^n_x||} - \frac{b^n_y}{||b^n_y||} || \leq
  \frac 1{||b^n_x||} \cdot 2 \cdot || b^n_x - b^n_y || \leq
  \frac 4{n^{k+1}} \cdot 2 \cdot \frac {n(n^k-1)}{n-1} \cdot D(x,y) = 8 \cdot \frac {n^k-1}{n^k(n-1)} \leq \frac 8 {n-1} < \varepsilon.
\]

We will show that there exists $S \in \mathcal{E}$ such that for each $x \in X$ we have $\supp(a_x) \subset S[x]$. This is the third condition of coarse property A. From Lemma~\ref{bnx} there exists $S \in \mathcal{E}$ such that for each $x \in X$ we have $\supp b^n_x \subset S[x]$. By the definition, $\supp a_x = \supp b^n_x$, which concludes the proof.
\end{proof}

\subsection{Further properties of coarse property C}

The goal of this section is to establish a union theorem for coarse property C. We will follow the program used by \cite{BD1} that has also been successfully applied in \cite{GTY1} and \cite{BM}. The basic strategy is to make use of a ``saturated union'' to combine two covers into a single cover that keeps some of the boundedness and separated properties of the original covers.  Our approach in the coarse category is similar to Grave's, c.f. \cite[Theorem 3.29]{Gra05}. 

\begin{definition}[\cite{Gra05}]
Let $\sU$ and $\sV$ be families of subsets of the coarse space $(X,\sE)$.  Let $V\in \sV$ and $L\in\sE$ be an entourage.  We define \[N_L(V,\sU) := V \cup \bigcup_{\substack{U\in\sU,\\ L\cap U\times V\neq \emptyset} }U.\]  The {\em $L$-saturated union of $\sV$ in $\sU$} is denoted $\sV \cup_L\sU$ and is given by $\sV \cup_L\sU := \{N_{L}(V,\sU) \mid V\in \sV\} \cup \{U\in \sU \mid L\cap \left(U\times V\right) = \emptyset,\ \forall  V\in \sV\}$. Observe that $\sV\cup_L\{\emptyset\}=\sV$ and $\{\emptyset\}\cup_L\sU=\sU$.
\end{definition}


\begin{proposition}\label{L-sat}
Let $(X,\sE)$ be a coarse space. Let $L$ be a symmetric entourage that contains $\Delta_X$. Put $K=L\cup \left( L\circ\Delta_{\sU}\circ L\right) \cup \left( L\circ\Delta_{\sU}\circ L\circ\Delta_{\sU}\circ L\right)$ and observe that this is an entourage. If $\sU$ is a uniformly bounded $L$-disjoint collection of subsets of $X$ and $\sV$ is a uniformly bounded $K$-disjoint collection of subsets of $X$, then $\sV \cup_L \sU$ is $L$-disjoint and uniformly bounded.
\end{proposition}
\begin{proof}
Put $D=\Delta_\sV\cup(\Delta_\sV\circ L\circ\Delta_\sU)$. Then, for any set $W\in N_L(V,\sU)$ one can take any $x\in V$ and see that $W\subset D[\{x\}]$. Thus $\sV \cup_L \sU$ is uniformly bounded, by Grave's characterization \cite[Proposition 3.10]{Gra05}.  

To show that $\sV \cup_L \sU$ is $L$-disjoint, let $A,B \in \sV \cup_L \sU$, with $A\neq B$.  We will proceed by cases.  

Case 1:  $A,B \in \{U\in \sU \mid L \cap (U\times V)=\emptyset,\ \forall V\in \sV\}$.  In this case $A$ and $B$ are distinct elements of $\sU$ and so we have that $L \cap A\times B = \emptyset$ since $\sU$ is $L$-disjoint.

Case 2: $A\in\{N_L(V,\sU)\mid V\in \sV\}$ and $B\in \{U\in \sU\mid L\cap (U\times V) = \emptyset,\ \forall V\in \sV\}$. Put $A=N_L(V_A,\sU)$ for a fixed $V_A\in\sV$ and observe that in particular $L\cap (V_A\times B)=\emptyset$. Thus, if $(a,b)\in A\times B\cap L$ then $a$ cannot be in $V_A$. Thus, it must be the case that $a\in U_a$ for some $U_a\in\sU$ such that $V_A\times U_a\cap L\neq\emptyset$. Thus, $a\times b\in L$ contradicts the fact that $\sU$ is $L$-disjoint. We conclude, therefore, that $A$ and $B$ are $L$-disjoint in this case.

Case 3: $A,B \in \{N_{L}(V,\sU)\mid V\in \sV\}$.   Let $A = N_{L}(V_A,\sU)$ and $B=N_{L}(V_B,\sU)$.  Since we are assuming that $A\neq B$, we have $V_A \neq V_B$. 

Suppose we have $a\in A$ and $b\in B$ with $(a,b)\in L$. Since $K$ contains $L$ and $\sV$ is $K$-disjoint, we know that $a$ and $b$ cannot belong to $V_A$ and $V_B$, respectively. Thus, in this case $(a,b)\neq L$.

If $a$ belongs to (say) $U_a\in\sU$ for which $V_A\times U_a\cap L\neq\emptyset$, and $b\in V_B$, then there is a $u_a\in U_a$ and a $v_a\in V_A$ so that $(v_a,b)=(v_a,u_a)\circ(u_a,a)\circ (a,b)\in L\circ\Delta_{\sU}\circ L$, which contradicts the fact that $\sV$ is $K$ disjoint. So, in this case $(a,b)\neq L$. There is a symmetric case (exchanging the roles of $A$ and $B$) that we omit.

Finally, suppose that there is a $U_a\in \sU$ and a $\sU_b\in\sU$ so that $U_a\times V_A\cap L\neq \emptyset$ and $U_b\times V_B\cap L\neq\emptyset$. To arrive at a contradiction, assume that $(a,b)\in L$. Then, we can find $u_a\in U_a$, $u_b\in U_b$, $v_a\in V_A$ and $v_B\in V_B$ so that $(v_a,u_a)\in L$ and $(u_b,v_b)\in L$. Thus, $(v_a,v_b)=(v_a,u_a)\circ(u_a,a)\circ(a,b)\circ(b,u_b)\circ (u_b,v_b)\in L\circ\Delta_{\sU}\circ L\circ\Delta_{\sU}\circ L$. We conclude that $(v_a,v_b)\in K$, which is a contradiction. 

Thus, in each case we find $A\times B\cap L=\emptyset$ as required.
\end{proof}

We begin by proving a simple finite union theorem following the techniques of \cite{Gra05}. Later we prove an infinite union theorem using different techniques.

\begin{theorem}
Let $(X,\sE)$ be a coarse space. Suppose that $X=X_1\cup X_2$, where each $X_i$ has the coarse structure it inherits from $X$. If $X_1$ and $X_2$ have coarse property C then so does $X$.
\end{theorem}
\begin{proof}
Let $L_1\subset L_2\subset L_3\subset \cdots$ be a sequence of symmetric entourages containing $\Delta_X$.  

Since $X_1$ has coarse property C, there exists a finite sequence $\sU_1, \sU_2, \ldots , \sU_{n_1}$ such that $\sU =\bigcup_{i=1}^{n_1} \sU_i$ covers $X_1$, each $\sU_i$ is uniformly bounded and each $\sU_i$ is $L_i$-disjoint.   Since $X_2$ has coarse property C, there exists a finite sequence $\sV_1, \sV_2, \ldots, \sV_{n_2}$ such that $\sV =\bigcup_{i=1}^{n_2} \sV_i$ covers $X_2$, each $\sV_i$ is uniformly bounded and each $\sV_i$ is  $K_i$-disjoint, where $K_i=L_i\cup \left(L_i\Delta_{\sU}L_i\right)\cup\left(L_i\Delta_{\sU}L_i\Delta_{\sU}L_i\right)$.

Put $n=\max\{n_1,n_2\}$ and for each $i=1,2,\ldots, n$, set $\sW_i = \sV_i \cup_{L_i}\sU_i$.  (Recall that even if one of the two families is empty, the saturated union is still defined.) We observe that since $\sU$ and $\sV$ cover $X_1$ and $X_2$ respectively, $\sW=\bigcup_{i=1}^n \sW_i$ covers $X$.

By Proposition \ref{L-sat}, we have that $\sW_i$ will be uniformly bounded and $L_i$-disjoint, for $i=1,2,\ldots \min\{n_1,n_2\}$. For higher indices $\sW_i$ is equal to $\sU_i$ (if $n_2>n_1$) or $\sV_i$ (if $n_1<n_2$). Thus, we have a finite collection of uniformly bounded subsets that cover $X$ and are $L_i$-disjoint for $i=1,2,\ldots,n$. We conclude that has coarse property C.
\end{proof}

Our next goal is to prove that certain infinite unions preserve coarse property C. Gromov \cite{Gr03} provided the first known examples of finitely generated groups without property A (see also \cite{Nowak}). Since (in the metric case) asymptotic property C implies property A \cite{Dr00}, we can conclude that there are countable discrete sets that do not have property C. Therefore we require some measure of control over the structure of the union in order to get a positive result on property C. As a first piece of this control we define uniform coarse property C. 

\begin{definition} Let $(X,\sE)$ be a coarse space. By a family of coarse subspaces $(X_\alpha,\sE_\alpha)$ we mean a collection of subsets $\{X_\alpha\}_\alpha$ of $X$ where the coarse structures $\sE_\alpha$ are obtained by restricting $\sE$ to $X_\alpha\times X_\alpha$. We say that the family of subspaces $X_\alpha$ has \emph{uniform coarse property C} if for any sequence $L_1\subset L_2\subset L_3\subset \cdots$ in $\sE$ there is a sequence $K_1\subset K_2\subset K_3\subset \cdots$ in $\sE$ and an $N \in \mathbb{N}$ such that for each $\alpha$ there exists a finite sequence $\sU_\alpha^1, \sU_\alpha^2, \ldots, \sU_\alpha^n$ with $n\leq N$ of subsets of $X_\alpha$ so that 
\begin{enumerate}
	\item for each $\alpha$, $\sU_\alpha=\bigcup_{i=1}^n \sU_\alpha^i$ covers $X_\alpha$;
	\item for each $\alpha, \Delta_{\sU_\alpha^i} \subset K_i$, (i.e. $\sU_\alpha^i$ is $K_i$-bounded); and
	\item for each $\alpha$ and for each $i$, $\sU_\alpha^i$ is $L_i$-disjoint.
\end{enumerate}
\end{definition}


\begin{theorem} \label{cPCUnion}
Suppose that $(X,\sE)$ is a coarse space that can be written as a union of coarse subspaces $X=\bigcup_{\alpha}X_{\alpha}$, where the family $X_\alpha$ has uniform coarse property C. Suppose moreover that for each entourage $K\in\mathcal{E}$ there is a subset $Y_K\subseteq X$ with coarse property C such that $\{X_\alpha\setminus Y_K\}_\alpha$ forms an $K$-disjoint collection. Then, $X$ has coarse property C.
\end{theorem}

\begin{proof}
 Let $L_1 \subseteq L_2 \subseteq \cdots$ be a sequence of entourages.  For each $\alpha$, choose families $\sU_\alpha^i$ of $L_i$-disjoint, $K_i$-bounded sets, where $\sU_\alpha = \cup_{i=1}^n \sU_\alpha^i$ is a cover of $X_\alpha$.  Let $\sU = \cup_\alpha \sU_\alpha$ and put $K = L_n\cup (L_n\Delta_\sU L_n)\cup (L_n\Delta_\sU L_n\Delta_\sU L_n)$.  Take $Y_K$ as in the statement of the theorem.

Since $Y_K$ has coarse property C, let $\sV^1, \sV^2, ..., \sV^k$ be $L_i\cup (L_i\Delta_\sU L_i)\cup (L_i\Delta_\sU L_i\Delta_\sU L_i)$-disjoint, uniformly bounded families of sets whose union covers $Y_K$.

Let $\overline{\sU_\alpha^i}$ denote the restriction of $\sU_\alpha^i$ to $X_\alpha \setminus Y_K$ and put $\overline{\sU^i} = \cup_\alpha \overline{\sU_\alpha^i}$.  Since $\overline{\sU_\alpha^i}$ are each $L_i$-disjoint and $X_\alpha \setminus Y_K$ are $K$-disjoint and thus $L_i$-disjoint $\forall i$, we have that $\overline{\sU^i}$ is $L_i$-disjoint.  We note that $\Delta_{\overline{\sU^i}} \subset K_i$, since each $\Delta_{\overline{\sU_\alpha^i}} \subset K_i$ and therefore $\overline{\sU^i}$ is uniformly bounded.

Now, set $\sW^i = \sV^i \cup_{L_i} \overline{\sU^i}$ for $i = 1,2\cdots,\max\{k,n\}$. (As before, this is defined even if one family is empty.)  By Proposition \ref{L-sat}, $\sW^i$ is $L_i$-disjoint and uniformly bounded.  Clearly, $\sW = 
\cup \sW^i$ covers $X$ and so $X$ has coarse property C.
\end{proof}

We can use the same techniques show that finite coarse asymptotic dimension is preserved by infinite unions with the analogous core structure.

\begin{definition} A family of coarse subspaces $(X_\alpha,\sE_\alpha)$ of the coarse space $(X,\sE)$ satisfies the inequality $\as_C X_\alpha \leq n$ {\em uniformly} if for any entourage $L$ there is an entourage $K$ such that for each $\alpha$ there exists a finite sequence $\sU_\alpha^1, \sU_\alpha^2, \ldots, \sU_\alpha^n$ so that 
\begin{enumerate}
	\item for each $\alpha$, $\sU_\alpha=\bigcup_{i=1}^n \sU_\alpha^i$ covers $X_\alpha$;
	\item for each $\alpha, \Delta_{\sU_\alpha} \subset K$, (i.e. $\sU_\alpha^i$ is $K$-bounded); and
	\item for each $\alpha$ and for each $i$, $\sU_\alpha^i$ is $L$-disjoint.
\end{enumerate}
\end{definition}

\begin{theorem} \label{cFADUnion}
Suppose that $X=\bigcup_{\alpha}X_{\alpha}$, where $\as_C X_\alpha \leq n$ uniformly and for each entourage $L\in\mathcal{E}$ there is a subset $Y_L\subseteq X$ with $\as_C Y_L \leq n$ such that $\{X_\alpha\setminus Y_L\}$ forms an $L$-disjoint collection. Then, $\as_C X \leq n$.
\end{theorem}

\begin{proof}
We will follow the techniques in \cite[Theorem 1]{BD1}. Let $L$ be an entourage.  For each $\alpha$, choose families $\sU_\alpha^1, \cdots,\sU_\alpha^n $ of $L$-disjoint, $K$-bounded sets, where $\sU_\alpha = \cup_{i=1}^n \sU_\alpha^i$ is a cover of $X_\alpha$.  Let $\sU = \cup_\alpha \sU_\alpha$ and put $M = L\cup (L\Delta_\sU L)\cup(L\Delta_\sU L\Delta_\sU L)$.  Take $Y_M$ as in the statement of the theorem.

Since $\as_C Y_M \leq n$, let $\sV^0, \sV^1, ..., \sV^n$ be $M$-disjoint, uniformly bounded families of sets whose union covers $Y_K$.

Let $\overline{\sU_\alpha^i}$ denote the restriction of $\sU_\alpha^i$ to $X_\alpha \setminus Y_K$ and put $\overline{\sU^i} = \cup_\alpha \overline{\sU_\alpha^i}$.  Since $\overline{\sU_\alpha^i}$ are each $L$-disjoint and $X_\alpha \setminus Y_M$ are $M$-disjoint and thus $L$-disjoint, we have that $\overline{\sU^i}$ is $L$-disjoint.  We note that $\Delta_{\overline{\sU^i}} \subset K$, since each $\Delta_{\overline{\sU^i_\alpha}} \subset K$ and therefore $\overline{\sU^i}$ is uniformly bounded.

Now, set $\sW^i = \sV^i \cup_{L} \overline{\sU^i}$ for $i = 0,1,\ldots,n$.  By Proposition \ref{L-sat}, $\sW^i$ is $L$-disjoint and uniformly bounded.  Clearly, $\sW = 
\cup \sW^i$ covers $X$ and so $\as_C X \leq n$.
\end{proof}

\section{Coarse Notions of Decomposition Complexity}

Following the ideas of the previous section we can translate notions of decomposition complexity to coarse spaces. Doing so, we obtain the following definition for a coarse version of finite decomposition complexity. We define three notions of decomposition complexity and are chiefly concerned with its finiteness: finite coarse decomposition complexity, finite weak coarse decomposition complexity, and straight finite coarse decomposition complexity.

Let $(X,\sE)$ be a coarse space; let $\sY$ be a family of coarse subspaces of $\sX$; let $L\in\sE$ be an entourage. An {\em $L$-decomposition} of $X$ over $\sY$ is a decomposition
\[X=X^0\cup X^1\qquad X^i=\bigsqcup_{L}X^i_{j}\] where each $X^i_{j}\in\sY$ and the union notation means that this union is $L$-disjoint in the sense that $X^i_{j}\neq X^i_{j'}$ implies $X^i_{j}\times X^i_{j'}\cap L=\emptyset$. Recall that we call the coarse family $\sY$ uniformly bounded if $\cup_{Y\in\sY}Y\times Y$ is an entourage.

Let $\sX$ and $\sY$ be families of coarse subspaces of a coarse space $(X,\sE)$. We say that the family $\sX$  admits an $L$-decomposition over $\sY$ if every $X\in\sX$ admits an $L$-decomposition over $\sY$. Thus, we may interpret a decomposition of a coarse space $X$ as a decomposition of the family $\{X\}$.

The decomposition game for the coarse space $X$ works as follows: Player 1 asserts an entourage $L_1$; Player 2 responds with a family $\sY_1$ and an $L_1$-decomposition of $\{X\}$ over $\sY_1$; then, Player 1 asserts another entourage $L_2$ and Player 2 responds with an $L_2$-decomposition of $\sY_1$ over a family $\sY_2$. The game continues in this way. It ends and Player 2 is declared the winner if at some finite stage the family chosen by Player 2 (over which the decomposition occurs) is uniformly bounded.

\begin{definition} The coarse space $X$ is said to have {\em finite coarse decomposition complexity} if Player 2 has a winning strategy in the coarse decomposition game for the family $\{X\}$, i.e., the process of choosing entourages $L_i$ and families $\sY_i$ with $L_i$-decompositions ends at a finite stage with a uniformly bounded family.
\end{definition}
\begin{definition}
The coarse space $X$ is said to have {\em straight finite coarse decomposition complexity} if for any sequence of controlled sets $L_1\subset L_2\subset\cdots$ there exists some finite sequence of families $\sY_1,\sY_2,\ldots,\sY_n$ of subsets of $X$ so that with $\sY_0=\{X\}$, we have $\sY_{i-1}\xrightarrow{L_i}\sY_i$ with uniformly $\sY_n$ bounded.
\end{definition}

As mentioned above, Dydak \cite{Dydak15} defined the notion of countable asymptotic dimension for metric spaces and showed that metric spaces with straight finite decomposition complexity have countable asymptotic dimension. Later, Dydak and Virk \cite{DV} showed that straight finite decomposition complexity and countable asymptotic dimension coincide for metric spaces. In the current paper we have not attempted to ``coarsify'' the notion of countable asymptotic dimension nor address its relation to our straight finite coarse decomposition complexity.


In the metric setting, the proof that finite asymptotic dimension implies finite decomposition complexity relies on embedding the space with finite asymptotic dimension into a product of trees. The construction of these universal spaces (products of trees) relies on a sequence of anti-\v{C}ech approximations, which do not seem to have a straightforward analog in the coarse setting (see also \cite[Definition 3.18]{Gra05}).  Following \cite{GTY1}, we weaken the definition of finite decomposition complexity in order to get a property that is clearly implied by finite asymptotic dimension. We do not resolve the issue of whether finite weak coarse decomposition complexity is distinct from finite coarse decomposition complexity. 

Let $(X,\sE)$ be a coarse space and let $\sY$ be a family of coarse subspaces of $X$. Given an entourage $L\in\sE$ and a positive integer $d$, we say that $X$ admits a {\em weak $(L,d)$-decomposition over $\sY$} if 
\[X=X^1\cup X^2\cup \cdots \cup X^d\] so that, for each $i=1,\ldots, d$, \[X^i=\bigsqcup_{L}X^i_{j}\]
where each $X^i_{j}\in\sY$ and the union is $L$-disjoint in the sense that $X^i_{j}\neq X^i_{j'}$ implies $X^i_{j}\times X^i_{j'}\cap L=\emptyset$. 

\begin{definition} The coarse space $X$ will be said to have {\em  finite weak coarse decomposition complexity} if the second player has a winning strategy in the weak coarse decomposition game.
\end{definition}

A simple consequence of this definition is the following proposition.

\begin{proposition}\label{cFADcwFDC}
Let $(X,\sE )$ be a coarse space with $\as_C X \le n$.  Then $(X,\mathcal{E})$ has finite weak coarse decomposition complexity.
\end{proposition}
\begin{proof}
Given an entourage $L\in \sE$, there exists a $L$-disjoint uniformly bounded families $\sU^0, \cdots, \sU^n$ whose union $\sU$ covers $X$. These families yield a weak $(L,n+1)$-decomposition of $X$ over $\sU$ and therefore the second player can win on the first phase of the weak coarse finite decomposition game.
\end{proof}

In the metric case, asymptotic property C implies straight finite decomposition complexity, \cite{DZ}. Translating these notions to the coarse setting yields the following results.

\begin{proposition}\label{cPCcsFDC}
Let $(X,\sE)$ be a coarse space with coarse property C.  Then $(X,\mathcal{E})$ has straight finite coarse decomposition complexity.
\end{proposition}

\begin{proof}
Let  $L_1\subset L_2\subset\cdots$ be a sequence of entourages. There exist families $\sU^1, \cdots, \sU^n$ such that $\sU = \cup_{i=1}^n\sU^i$ covers X and each $\sU^i$ is uniformly bounded and $L_i$-disjoint.

We define $\sY_1= \{X\setminus\cup_{U\in\sU^1}U\} \bigcup \sU^1$.  Then, with $X^0=X\setminus\cup_{U\in\sU^1}U$ and $X^1=\cup_{U\in \sU^1}U$ we see that $X$ admits an $L_1$-decomposition over $\sY_1$.

For $j>1$, define $\sY_j=\{\bigcap_{i=1}^jX\setminus\cup_{U\in\sU^i}U\}\cup \sU^1\cup\sU^2\cup\cdots\cup\sU^j$. It is clear that in this way, $\sU^n$ will be uniformly bounded. It remains to show only that for each $Z\in\sY_{j-1}$ (for $j=2,3,\cdots,n$), $Z$ admits an $L_j$ decomposition over $\sY_j$. To this end, suppose we are given a set $Z\in\sY_{j-1}$. If $Z$ is equal to some element $U$ in a family $\sU^k$, with $1\le k\le j-1$, then this $Z$ is also in $\sY_j$ and so it is trivially $L_j$-decomposable over $\sY_j$. Otherwise, $Z=\bigcap_{i=1}^{j-1}X\setminus\cup_{U\in\sU^i}U$. This set can be written as $X^0\cup X^1$, where $X^0=\bigcap_{i=1}^jX\setminus\cup_{U\in\sU^i}U$ and $X^1=\cup_{U\in \sU^j}U$ and this union is an $L_j$-disjoint union.
\end{proof}

\begin{corollary} Let $(X,\sE)$ be a coarse space such that $\as_CX \le n$.  Then $(X,\mathcal{E})$ has straight finite coarse decomposition complexity. \qed
\end{corollary}


\begin{figure}\begin{center}\includegraphics{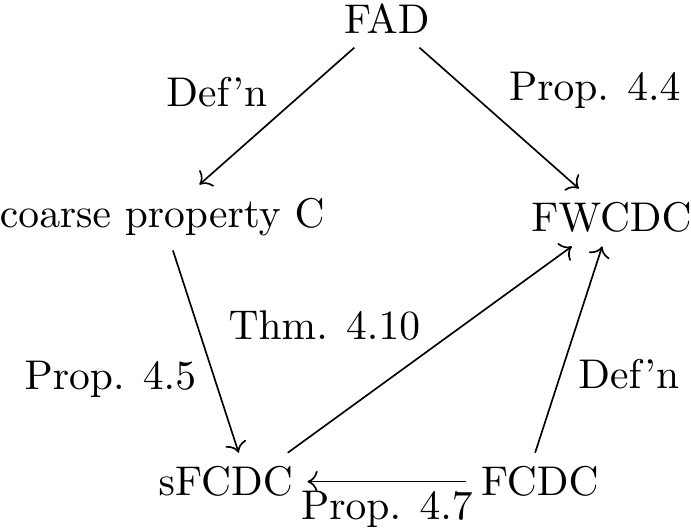}
\end{center}
\caption{A diagram indicating the implications among the coarse properties considered in this paper with references to the proofs. Here FAD means finite asymptotic dimension, sFCDC means straight finite coarse decomposition complexity, FCDC means finite coarse decomposition complexity, and FWCDC means finite weak coarse decomposition complexity.}
\end{figure}

\begin{proposition}\label{cFDCcsFDC}
Let $(X,\sE)$ be a coarse space such that $X$ has finite coarse decomposition complexity.  Then $(X,\mathcal{E})$ has straight finite coarse decomposition complexity.
\end{proposition}
\begin{proof}
Given a sequence of entourages $L_1\subset L_2\subset\cdots$, we can play the decomposition game where Player 1 always gives the next entourage in the sequence on their turn. Since $X$ has finite coarse decomposition complexity, Player 2 has a winning strategy, which will provide a finite sequence $\sY_1,\sY_2,\ldots,\sY_n$ so that $\sY_{i-1}\xrightarrow{L_i}\sY_i$ with $\sY_n$ bounded.
\end{proof}

The next lemma allows us a measure of control on composing weak decompositions.

\begin{lemma}\label{Lemma4.8}
Let $\sX,\sY,$ and $\sZ$ be families of subspaces of a coarse space $(W,\sE)$. Let $L_1$ and $L_2$ be entourages. Then if $\sX$ admits a weak $(L_1,d_1)$-decomposition over $\sY$ and $\sY$ admits a weak $(L_2,d_2)$-decomposition of $\sZ$, then $\sX$ admits a weak $(L_1\cap L_2,d_1\cdot d_2)$-decomposition over $\sZ$.  
\end{lemma}
\begin{proof} 
Although the notation is cumbersome, the idea is very straightforward as we just decompose each set $X\in \sX$ over $\sY$ and then over $\sZ$, then combine these decompositions.

Given any $X \in \sX$, we can write $X=X^1\cup \ldots \cup X^{d_1}$ where $X^i=\bigsqcup_{L_1}Y^i_{j}$, each $Y^i_{j}\in\sY$ and whenever $j\neq j'$ we have that $(Y^i_{j}\times Y^i_{j'})\cap L_1 = \emptyset$. Since $Y^i_{j}\in\sY$, for each $i$ and $j$ we can write $Y^i_{j}=Y^{i,1}_{j}\cup\ldots \cup Y^{i,d_2}_{j}$ where $Y^{i,k}_{j}=\bigsqcup_{L_2}Z^{i,k}_{j,\ell}$, each $Z^{i,k}_{j,\ell}\in\sZ$ and whenever $\ell\neq \ell'$ then we have that $(Z^{i,k}_{j,\ell}\times Z^{i,k}_{j,\ell'})\cap L_2 = \emptyset$. 

Therefore, we can write $X = \cup_{i=0}^{d_1} \cup_{k=1}^{d_2} \bigcup_{j,\ell} Z^{i,k}_{j,\ell}$. We will proceed by cases to show that the collection $\{Z^{i,k}_{j,\ell}\}_{j,\ell}$ is $L_1\cap L_2$-disjoint. We take two distinct elements of that collection: $Z^{i,k}_{j,\ell}$ and $Z^{i,k}_{j',\ell'}$.

Case 1: $j\neq j'$. In this case, we have that $Z^{i,k}_{j,\ell}\subset Y^{i}_j$ and $Z^{i,k}_{j',\ell'}\subset Y_{j'}^{i}$ and so therefore $(Z^{i,k}_{j,\ell}\times Z^{i,k}_{j',\ell'})\cap L_1 = \emptyset$. So, in particular $(Z^{i,k}_{j,\ell}\times Z^{i,k}_{j',\ell'})\cap (L_1\cap L_2) = \emptyset$.

Case 2: $j=j'$. Since these elements are distinct, we must have that $\ell\neq \ell'$. Thus $(Z^{i,k}_{j,\ell}\times Z^{i,k}_{j,\ell'})\cap L_2 = \emptyset$ and therefore $(Z^{i,k}_{j,\ell}\times Z^{i,k}_{j',\ell'})\cap (L_1\cap L_2) = \emptyset$.
\end{proof}

 As a consequence, we can simplify the definition of finite weak coarse decomposition complexity.

\begin{proposition}
Let $\sX$ be a family of coarse subspaces of a coarse space $(X,\sE)$. If $\sX$ has finite weak coarse decomposition complexity then the weak decomposition game can be won by Player 2 on the first turn.
\end{proposition}
\begin{proof}
Suppose that Player 1 asserts an entourage $L_1$. Instead of responding with a single family $\sY_1$ and an $(L_1,d_1)$-decomposition of $X$ over the family $\sY_1$, Player 2 plays the decomposition game on her own. She again takes the entourage $L_1$ and a family $\sY_2$ so that there is an $(L_1, d_2)$-decomposition of $\sY_2$ over $\sY_1$. She continues in this way. Because the family $\sX$ is assumed to have finite weak coarse decomposition complexity, this process must stop at some finite point at which Player 2 has found an $(L_1,d_k)$-decomposition of $\sY_{k-1}$ over $\sY_k$, with $\sY_k$ uniformly bounded. By Lemma \ref{Lemma4.8}, Player 2 can respond to the initial challenge with an $(L_1, d_1\cdot\cdots \cdot d_k)$-decomposition of $\sX$ over the uniformly bounded family $\sY_k$ and therefore, Player 2 can win on turn one. 
\end{proof} 

\begin{theorem}\label{csFDCcwFDC}
Let $(X,\sE)$ be a coarse space such that $X$ has straight finite coarse decomposition complexity.  Then $(X,\mathcal{E})$ has finite weak coarse decomposition complexity.
\end{theorem}
\begin{proof} 
Given $L\in \mathcal{E}$, let $L_1\subset L_2\subset\cdots$ be the sequence of controlled sets with $L_i=L$ for all $i$.  Then there exists some finite sequence $\sY_1,\sY_2,\ldots,\sY_n$ so that $\sY_{i-1}\xrightarrow{L}\sY_i$ with $\sY_n$ bounded. Then by the above, $X$ admits an $(L,2^n)$-decomposition over $\sY_n$ and $\sY_n$ is bounded. Therefore, $X$ has finite weak coarse decomposition complexity.
\end{proof}

\subsection{Permanence results}
We want to establish some basic permanence results for these coarse notions. We begin by showing that these properties are inherited by subspaces and are coarse invariants. These permanence results are often of the form: let $Y$ be some space obtained from $X$. If $X$ has the property $\mathcal{P}$ then $Y$ has the property $\mathcal{P}$ where $\mathcal{P}$ is finite weak coarse decomposition complexity, finite coarse decomposition complexity or straight finite coarse decomposition complexity. We keep the approach uniform by splitting each statement into three separate statements (one for each version of finite decomposition complexity).

We omit the obvious proof of the following result.

\begin{proposition}
If $Y\subset X$ and $Y$ has the coarse structure inherited from $X$, then
\begin{enumerate}
\item if $X$ has finite weak coarse decomposition complexity then $Y$ has finite weak coarse decomposition complexity;
\item if $X$ has finite coarse decomposition complexity then $Y$ has finite coarse decomposition complexity; and
\item if $X$ has straight finite coarse decomposition complexity then $Y$ has straight finite coarse decomposition complexity.
\end{enumerate}\qed
\end{proposition}

%

In a recent article, Dydak and Virk investigate the extent to which properties such as straight finite decomposition complexity and property C are preserved by certain classes of functions \cite{DV}. In the current paper we did not investigate the extent to which their methods could be applied to these coarse properties. Instead, we content ourselves with considering a very strong assumption on a map between coarse spaces.

\begin{theorem}\label{FDC-CI} Let $(X,\sE)$ and $(Y,\sF)$ be coarsely equivalent coarse spaces. Then, 
\begin{enumerate}
\item  $X$ has finite weak coarse decomposition complexity if and only if $Y$ has finite weak coarse decomposition complexity;
\item $X$ has finite coarse decomposition complexity if and only if $Y$ has finite coarse decomposition complexity; and
\item  $X$ has straight finite coarse decomposition complexity if and only if $Y$ has straight finite coarse decomposition complexity.
\end{enumerate} 
\end{theorem}
\begin{proof} 
(1) Let $f:X\to Y$ be a coarse equivalence and suppose that $Y$ has finite weak coarse decomposition complexity.   To construct  a winning strategy for the decomposition game for $X$, we will play a parallel game for $Y$ as follows:  For the first stage, Player 2 is given a controlled set $L_1$ and we take as our initial controlled set in the parallel game to be $K_1 = (f\times f)(L_1)$.  Then, since $Y$ has finite weak coarse decomposition complexity, we can find a family $\sY_1$ and a $(K_1,d_1)$-decomposition of $Y$ over $\sY_1$.  We claim that $X$ has an $(L_1,d_1)$-decomposition over the family $\sX_1 = \{(f\times f)^{-1}(B)\mid B\in \sY_1\}$.

We have \[Y=Y^1\cup \cdots \cup Y^{d_1} \qquad Y^i=\bigsqcup_{K_1}Y^i_{j}\] where each $Y^i_{j}\in\sY_1$ and the union is $K_1$-disjoint in the sense that $Y^i_{j}\neq Y^i_{j'}$ implies $Y^i_{j}\times Y^i_{j'}\cap K_1=\emptyset$. Then, with $X^i = (f\times f)^{-1}(Y^i)$ and $X^i_{j} = (f\times f)^{-1}(Y^i_{j})$, we see that \[X=X^1\cup \cdots \cup X^{d_1},\]
and each $X^i_{j}\in\sX_1$. Moreover since $X^i_{j}\neq X^i_{j'}$ implies $Y^i_{j}\neq Y^i_{j'}$ and $X^i_{j}\times X^i_{j'}\cap L_1\subset (f\times f)^{-1}(Y^i_{j}\times Y^i_{j'}\cap K_1) = (f\times f)^{-1}(\emptyset) = \emptyset$ we find that \[X^i=\bigsqcup_{L_1}X^i_{j}.\] Thus, $X$ has an $(L_1,d_1)$-decomposition over $\sX_1$.

On the $i$-th stage, Player 2 has found a family $\sX_{i-1}$ and is given a controlled set $L_i$.  We note that $\sY_{i-1} = \{(f\times f)(A)\mid A\in\sX_{i-1}\}$ by the previous construction and define $K_i = (f\times f)(L_i)$.  As before, we can find a family $\sY_i$ and a $(K_i,d_i)$-decomposition of $\sY_{i-1}$ over $\sY_i$, and again, it is true that $\sX_{i-1}$ has an $(L_i,d_i)$-decomposition over the family $\sX_i = \{(f\times f)^{-1}(B)\mid B\in \sY_i\}$.

Since $Y$ has finite weak coarse decomposition complexity, after a finite number of stages (say $k$), $\sY_k$ will be uniformly bounded, and therefore so will $\sX_k$ since $f$ is a coarse equivalence.  Therefore, Player 2 has a winning strategy for the weak coarse decomposition game for $X$, i.e. $X$ has finite weak coarse decomposition complexity.

(2) We note that since $d_i$ for $X$ in any given stage is the same as $d_i$ for $Y$, this also proves that finite coarse decomposition complexity is a coarse invariant, since in that case $d_i=2$ for all $i$.

(3) If $Y$ is assumed to have straight finite coarse decomposition complexity, then we begin with $L_1\subset L_2\subset\cdots$ a sequence of entourages in $\sE$.  Then since $f$ is bornologous, we have that $K_i = (f\times f)(L_i)$ is a sequence of entourages in $\sY$ such that $K_1\subset K_2\subset\cdots$.  Therefore, since $Y$ has straight finite coarse decomposition complexity, there exist $\sY_1, \sY_2\cdots, \sY_n$ such that $\sY_{i-1} \xrightarrow{K_i}\sY_i$ with $\sY_n$ bounded.

Let $\sX_i = \{(f\times f)^{-1}(B)\mid B\in\sY_i\}$.  As above, we use the fact that $f$ is a coarse equivalence to conclude that $\sX_{i-1} \xrightarrow{L_i}\sX_i$ and also that $X_n$ is bounded.  Therefore, $X$ has straight finite coarse decomposition complexity.
\end{proof}


We now wish to show that certain unions and products preserve these properties (cf.~\cite[Theorem 3.1.4]{GTY2} and \cite[Theorem 5.2]{BM}). Before proving our fibering theorem (Theorem \ref{fibering}), we observe that coarse decompositions are preserved by inverse images in the following sense.

\begin{remark}\label{remark4.13} Let $f:X\to Y$ be a coarse map between the coarse spaces $(X,\sE)$ and $(Y,\sF)$. Suppose $L\in\sE$ and put $K=(f\times f)(L)$. Then, since $f$ is bornologous, we see that $K\in\sF$. Let $\sV_1$ and $\sV_2$ be families of subsets of $Y$ for which there is a $K$-decomposition of $\sV_1$ over $\sV_2$. Then, with $\sU_i=\{f^{-1}(V)\mid V\in\sV_i\}$, it follows from basic set operations that $\sU_1$ admits an $L$-decomposition over $\sU_2$. 
\end{remark} 

\begin{theorem}\label{fibering}
Let $(X,\sE)$ and $(Y, \sF)$ be coarse spaces and suppose that $f:X\to Y$ is a coarse map. 
\begin{enumerate}
\item Suppose that for every uniformly bounded family $\sV$ of subsets of $Y$, the family $f^{-1}(\sV)$ and $Y$ itself have finite weak coarse decomposition complexity; then $X$ has finite weak coarse decomposition complexity.
\item Suppose that for every uniformly bounded family $\sV$ of subsets of $Y$, the family $f^{-1}(\sV)$ and $Y$ itself have finite coarse decomposition complexity; then $X$ has finite coarse decomposition complexity.
\item Suppose that for every uniformly bounded family $\sV$ of subsets of $Y$, the family $f^{-1}(\sV)$ and $Y$ itself have straight finite coarse decomposition complexity; then $X$ has straight finite coarse decomposition complexity.
\end{enumerate}
\end{theorem}
\begin{proof} (1) The proof for statement (1) follows from the proof of statement (2) with the small change that in each step we are considering a weak decomposition (i.e. a $(L_i, d_i)$ decomposition). With this small change, the proof follows the exact same scheme and so we omit it.

(2) We must show that Player 2 (P2) has a winning strategy in the coarse decomposition game for $X$. To this end, suppose Player 1 (P1) asserts the entourage $L_1\in\sE$. Put $K_1=(f\times f)(L_1)$. Since $f$ is bornologous, $K_1\in\sF$. Put $\sV_0=\{Y\}$. Since $Y$ has finite coarse decomposition complexity by assumption, P2 has a winning strategy in the coarse decomposition game. Thus, P2 responds with a family $\sV_1$ of subsets of $Y$ and a $K_1$ decomposition of $Y$ over $\sV_1$. Then, P1 asserts another entourage and the game continues. Since P2 has a winning strategy, at some finite stage the game ends with P2 returning a uniformly bounded family (say) $\sV_m$ and a $K_m$ decomposition of $\sV_{m-1}$ over $\sV_m$, where $K_m$ is an entourage asserted by P1.

Now since $\sV_m$ is uniformly bounded, the collection of subsets $f^{-1}(\sV_m)$ of $X$ has finite coarse decomposition complexity by assumption. Thus, P2 has a winning strategy in the coarse decomposition game for $f^{-1}(\sV_m)$. Thus, no matter what choice of entourage $L_{m+1}$ P1 makes, P2 can respond with a family of subsets $\sU_{m+1}$ of $X$ and a $L_{m+1}$-decomposition of $f^{-1}(\sV_m)$ over $\sU_{m+1}$. This continues until at some finite stage, P2 responds with a uniformly bounded family (say) $\sU_{m+n}$ and a $L_{m+n}$ decomposition of the previous family over this one. 

By Remark \ref{remark4.13}, the families $\sU_i=\{f^{-1}(V)\mid V\in \sV_i\}$ for $i=1,2,\ldots, m$, provide a sequence of coarse $L_i$-decompositions for $X$ that end in the family $\sU_m=f^{-1}(\sV_m)$. Thus, by concatenating this strategy with the strategy for $f^{-1}(\sV_m)$,  (the sets of the form $\sU_{m+j}$, with $j>0$) we have shown how P2 can win the coarse decomposition game for $X$. Thus, $X$ has finite coarse decomposition complexity.

(3) The proof of statement (2) can be modified to prove the statement in the case of straight finite decomposition complexity. Indeed, the only change is that P1 gives the entire sequence of entourages $L_1\subset L_2\subset \cdots$ at the onset of the game. Thus, P2 responds to the sequence by finding a straight finite decomposition sequence of $\{Y\}$ over the entourages $(f\times f)(L_i)$ that ends in some uniformly bounded family $\sV$ of subsets of $Y$. This yields a strategy for $X$ that can be concatenated with the straight decomposition over the inverse image $f^{-1}(\sV)$ as in part (2).
\end{proof}

\begin{remark} \label{bornologous} We would like to note that we did not require the map in the previous proposition to be coarse. It is enough to assume it is bornologous.\end{remark}

To show that straight finite coarse decomposition is preserved by products, we require a basic result on products with bounded sets. We use the standard bounded coarse structure, defined by Grave, \cite{Gra05}: if $(X_i,\sE_i)$ is a finite collection of coarse spaces, then the product coarse structure is given by $\{E\subset (X_1\times \cdots \times X_k)^2\mid (p_i\times p_i)(E)\in \sE_i\ \forall\ i\}$, where $p_i$ denotes the projection to the $i$-th factor.

\begin{proposition}\label{bounded}
Let $(X,\sE)$ and $(Y,\sF)$ be coarse spaces. Let $B$ be a bounded subset of $X$. Give $B$ the subspace coarse structure it inherits from $X$ and give $B\times Y$ the product coarse structure. Then $B\times Y$ is coarsely equivalent to $Y$. 
\end{proposition}
\begin{proof}
Let $b_0$ be any element of $B$. Define $g:Y\to B\times Y$ by $g(y)=(b_0,y)$. It is clear that $f$ is a coarsely uniform embedding. Let $f:B\times Y \rightarrow Y: (b,y) \mapsto y$.  We will show this is a coarsely uniform embedding.  

Since $f(B\times Y)$ is the same map as the projection to the second factor $p_2$ in the definition of the product coarse structure, we see that $f$ is bornologous. If $C\subset Y$ is bounded, then clearly $f^{-1}(C)=B\times C$ is bounded in the product structure and so $f$ is proper. Let $K$ be an element of $\sF$. Since $B$ is bounded, $B\times B\in\sE$. Thus $(p_1\times p_1)(f^{-1}(K))=B\times B$, which is an entourage (in $\sE$), while $(p_2\times p_2)(f^{-1}(K))=K$, which is an element of $\sF$. Thus, $f^{-1}(K)$ is an entourage in the product coarse structure on $X\times Y$. Therefore, $f$ is a coarsely uniform embedding. 

To finish the proof we must show that the compositions in both directions are close to the identity. It is clear that $f\circ g:Y\to Y$ is precisely the identity. On the other hand, $g\circ f: B\times Y\to B\times Y$ is close to the identity, since $\{(g\circ f)(b\times y),(b\times y)\mid b\times y\in B\times Y\}=\{(b_0\times y),(b\times y)\mid b\times y\in B\times Y\}$. This is in the product coarse structure since $(p_1\times p_1)(\{(b_0\times y),(b\times y)\mid b\times y\in B\times Y\})=b_0\times b\in B\times B$ and 
$(p_2\times p_2)(\{(b_0\times y),(b\times y)\mid b\times y\in B\times Y\})=y\times y$ which is the diagonal and so is an entourage.
\end{proof}

Now we are in a position to prove that these properties are preserved by direct products.

\begin{theorem}\label{cFDCdirect}
Let $(X,\sE)$ and $(Y, \sF)$ be coarse spaces. Let $Z=X\times Y$ with the product coarse structure. \begin{enumerate}
\item If $X$ and $Y$ have finite weak coarse decomposition complexity then so does $Z$.
\item If $X$ and $Y$ have finite coarse decomposition complexity then so does $Z$.
\item If $X$ and $Y$ have straight finite coarse decomposition complexity then so does $Z$.
\end{enumerate}
\end{theorem}
\begin{proof}
Let $f:Z\to X$ be the projection map. We note that $f$ is bornologous. If $\sV$ is a collection of subsets of $X$ that is uniformly bounded, then (by the definition) $\{V\times V\mid V\in \sV\}\in\sE$. Thus, if $A\in f^{-1}(\sV)$, then $A=V\times Y$ for some bounded $V\in\sV$. Thus, by Proposition \ref{bounded}, we see that each $A\in f^{-1}(\sV)$ has the same version of coarse decomposition complexity as $Y$. We apply Theorem \ref{fibering} and Remark \ref{bornologous} to complete the proof. 
\end{proof}

As was the case for property C, these notions of finite decomposition complexity are preserved by certain types of unions. 
We consider the case where $X$ can be expressed as a union of a collection of spaces with the property that for each entourage $L$ there is a ``core'' space $Y_L$ such that removing this core from the families leaves the families $L$-disjoint. We will be following the scheme we used to prove the corresponding results for asymptotic dimension and property C, adapting it for use in the coarse case. In this situation, however, we do not require a separate uniform version of the property.

\begin{theorem}\label{cFDCUnion}
Let $X = \cup \sX$. If for each entourage $L$, there exists $Y_L \subseteq X$ such that $\{X_\alpha\setminus Y_L\} = \sX_L$ forms an $L$-disjoint collection, then
\begin{enumerate}
\item  if $\sX$ has finite weak coarse decomposition complexity and $Y_L$ has finite weak coarse decomposition complexity $\forall L$, then $X$ has finite weak coarse decomposition complexity;
\item if $\sX$ has straight finite coarse decomposition complexity and $Y_L$ has finite coarse decomposition complexity $\forall L$, then $X$ has finite coarse decomposition complexity; and
\item  if $\sX$ has straight finite coarse decomposition complexity and $Y_L$ has straight finite coarse decomposition complexity $\forall L$, then $X$ has straight finite coarse decomposition complexity.
\end{enumerate}
\end{theorem}
\begin{proof} 
(1) Again, the only difference between (1) and (2) is the number of families into which the decompositions split. Thus, we omit the proof in this case.

(2) Suppose we are given an entourage $L$.  We consider the family $\sY = \{Y_{L}\} \cup \sX_{L}$ and write $X = X_0 \cup X_1$, where $X_0 = Y_{L}$ and $X_1 = \cup\sX_{L}$.  Since $X_0$ is a single element of the family, it is $L$-disjoint and we have that $\sX_{L}$ is $L$-disjoint by assumption, so therefore we have an $L$-decomposition of $X$ over $\sY$.

Now, $\sX$ has finite coarse decomposition complexity and thus $\sX_{L}$ has finite coarse decomposition complexity.  Therefore, since $\sY_{L}$ also has finite coarse decomposition complexity, Player 2 has a winning strategy for each element of $\sY$ and therefore, $X$ has finite coarse decomposition complexity.

(3) We use a different strategy here and follow the techniques in \cite{DZ}, Theorem 3.6.  Given $L_1\subseteq L_2 \subseteq \cdots$ a sequence of entourages, then we consider the family $\sY_1 = \{Y_{L_1}\} \cup \sX_{L_1}$ and write $X = X_0 \cup X_1$, where $X_0 = Y_{L_1}$ and $X_1 = \cup\sX_{L_1}$.  Since $X_0$ is a single element of the family, it is $L_1$-disjoint  and we have that $\sX_{L_1}$ is $L_1$-disjoint by assumption, so therefore we have an $L_1$-decomposition of $X$ over $\sY_1$.

Now, $\sX$ has straight finite coarse decomposition complexity and thus $\sX_{L_1}$ has straight finite coarse decomposition complexity.  Therefore, since $\sY_{L_1}$ also has straight finite coarse decomposition complexity, we have a natural number $n$ and families $\sY_i, i=2,3,\cdots,n$ such that $\sY_{i-1}\xrightarrow{L_i}\sY_i$ for $i=2,\cdots,n$ and $\sY_n$ is a bounded family.  Therefore, $X$ has straight finite coarse decomposition complexity.

\end{proof}

\begin{corollary} Let $X=X_1\cup X_2$ be a coarse space. 
\begin{enumerate}
\item If $X_1$ and $X_2$ have finite weak coarse decomposition complexity then $X$ does.
\item If $X_1$ and $X_2$ have finite coarse decomposition complexity then $X$ does.
\item If $X_1$ and $X_2$ have straight finite coarse decomposition complexity then $X$ does.
\end{enumerate} 
\end{corollary}

\begin{proof} For any entourage $L$, take $Y_L=X_2$. 
\end{proof}

If one tries to play the decomposition game with coarse property C, then one obtains a coarse version of Dranishnikov and Zarichnyi's game-theoretic property C, \cite{DZ}.

\begin{definition} The coarse space $(X,\sE)$ has {\em game-theoretic coarse property C} if there is a winning strategy for Player 2 in the following game. Player 1 selects an entourage $L_1$ and Player 2 finds a uniformly bounded family $\sU_1$ of sets that are $L_1$-disjoint. Then, Player 1 gives an entourage $L_2$ and Player 2 responds with an $L_2$-disjoint, uniformly bounded family $\sU_2$. The game ends and Player 2 wins if there is some $k$ for which $\sU=\cup_{i=1}^k\sU_k$ covers $X$.
\end{definition}

As in the case with Dranishnikov and Zarichnyi's metric version, the attempt to define game-theoretic coarse property C gives rise to precisely the same class of spaces with coarse asymptotic dimension 0.

\begin{proposition} A coarse space $(X,\sE)$ has game-theoretic coarse property C if and only if $\as_C(X,\sE)=0$.
\end{proposition}

\begin{proof} If $\as_C(X,\sE)=0$, then it is clear that $X$ the game ends in one step regardless of what $L$ is asserted by Player 1.

On the other hand, suppose that $X$ has game-theoretic coarse property C and there is some entourage $L$ for which $X$ has no uniformly bounded $L$-disjoint cover. Then, Player 1 selects any entourage $L_1$ that properly contains $L$. Player 2 responds with a uniformly bounded family $\sU_1$ that is $L_1$-disjoint. Player 1 then responds with the entourage $L_2=L_1\Delta_{\sU_1}L_1\Delta_{\sU_1}L_1.$ Player 2 responds with $\sU_2$ and Player 1 gives $L_3=L_2\Delta_{\sU_2}L_2\Delta_{\sU_2}L_2.$ This continues until at some point Player 2 returns $\sU_k$ so that the family $\sU=\sU_1\cup\cdots\cup\sU_k$ covers $X$. 

We use the saturated union. 
Set $\sV_{k-1}=\sU_k\cup_{L_{k-1}}\sU_{k-1}$. Then, this family is uniformly bounded and $L_{k-1}$ disjoint by Proposition \ref{L-sat}. Also, it covers $\sU_k\cup\sU_{k-1}$. Next, put $\sV_{k-2}=\sV_{k-1}\cup_{L_{k-2}}\sU_{k-2}$ and so on. Finally, one obtains a single family $\sV_1$ that is uniformly bounded, $L_1$-disjoint, and covers $X$. This contradicts the choice of $L$.
\end{proof}

\subsection{Open Questions}
We conclude with a few open questions. 
\begin{question}
Does straight finite coarse decomposition complexity imply coarse property A?
\end{question}

As we mentioned above, there is a second line of inquiry regarding the relationship between finite coarse asymptotic dimension and finite coarse decomposition complexity.

\begin{question}
If a coarse space $X$ has finite coarse asymptotic dimension and an anti-\v{C}ech approximation, does it have finite coarse decomposition complexity?
\end{question}

Alternately, we could use the approach of the proof in the metric case.

\begin{question}
If $X$ is a coarse space with finite asymptotic dimension, does it embed into a product of trees?
\end{question}

\section*{Acknowledgements} The authors are sincerely grateful to the anonymous referee for many helpful suggestions, a very close and careful reading of the manuscript, and for pointing our several helpful references (including references to countable asymptotic dimension). The first author would like to thank Piotr Nowak for very helpful remarks concerning our version of property A.

This work began when the first author visited Warsaw University and a stayed at IMPAN. Additionally much work was completed while the first author was a New Directions Visiting Professor the Institute for Mathematics and its Applications. Bell would like to thank Warsaw University, IMPAN, and the IMA for their support.

\bibliographystyle{alpha_initials_nodash}
\bibliography{database.bib}

\end{document}